\input ./style/arxiv-vmsta.cfg
\documentclass[numbers,compress,v1.0.1]{vmsta}

\volume{2}
\issue{2}
\pubyear{2015}
\firstpage{147}
\lastpage{164}
\doi{10.15559/15-VMSTA28}

\startlocaldefs
\newcommand{\rrvert}{\vert}
\newcommand{\llvert}{\vert}
\allowdisplaybreaks
\endlocaldefs

\newtheorem{thm}{Theorem}
\newtheorem{lemma}{Lemma}

\newtheorem{prop}{Proposition}

\theoremstyle{definition}

\newtheorem{remark}{Remark}

\newcommand{\R}{{\mathbb R}}
\newcommand{\abs}[1]{\llvert  #1\rrvert }
\newcommand{\pair}[1]{\left\langle #1\right\rangle}

\newcommand{\set}[1]{\left\{#1\right\}}
\newcommand{\ind}[1]{\mathbf{1}_{#1}}
\newcommand{\ex}[1]{\mathsf{E}\left[#1\right]}

\begin{document}
\begin{frontmatter}

\title{Construction of maximum likelihood estimator in the mixed
fractional--fractional
Brownian motion model with double long-range dependence}

\author{\inits{Yu.}\fnm{Yuliya}\snm{Mishura}\corref{cor1}}\email
{myus@univ.kiev.ua}
\cortext[cor1]{Corresponding author.}
\author{\inits{I.}\fnm{Ivan}\snm{Voronov}}\email{ivan.voronov@univ.kiev.ua}

\address{Taras Shevchenko National University of Kyiv, Volodymyrska str. 64, 01601,
Kyiv, Ukraine}

\markboth{Yu. Mishura, I. Voronov}{Construction of MLE}

\begin{abstract}
We construct an estimator of the unknown drift parameter $\theta\in\R$
in the linear model
\[
X_t=\theta t+\sigma_1 B^{H_1}(t)+ \sigma_2 B^{H_2}(t),\;t\in[0,T],
\]
where $B^{H_1}$ and $B^{H_2}$ are two independent fractional Brownian
motions with Hurst indices $H_1$ and $H_2$ satisfying the condition
$\frac{1}2 \leq H_1<H_2<1.$ Actually, we reduce the problem to the
solution of the integral Fredholm equation of the 2nd kind with
a~specific weakly singular kernel depending on two power exponents. It
is proved that the kernel can be presented as the product of a bounded
continuous multiplier and weak singular one, and this representation
allows us to prove the compactness of the corresponding integral
operator. This, in turn, allows us to establish an
existence--uniqueness result for the sequence of the equations on the
increasing intervals, to construct accordingly a sequence of
statistical estimators, and to establish asymptotic consistency.
\end{abstract}


\received{17 June 2015}
\revised{7 July 2015}
\accepted{7 July 2015}
\publishedonline{20 July 2015}
\end{frontmatter}

\section{Introduction}\label{intro}
Consider the continuous-time linear model
\begin{equation}
\label{model} X(t)=\theta t+\sigma_1 B^{H_1}(t)+
\sigma_2 B^{H_2}(t),\;t\in[0,T],
\end{equation}
where $B^{H_1}$ and $B^{H_2}$ are two independent fractional Brownian
motions with different Hurst indices $H_1$ and $H_2$ defined on some
stochastic basis $(\varOmega, \mathfrak{F}, (\mathfrak{F})_t, t\geq0,
\mathsf{P}). $ We assume that the filtration is generated by these
processes and completed by $\mathsf{P}$-negligible sets of $\mathfrak{F}_0$.

Recall that the fractional Brownian motion (fBm) $B^H_t,\ t\geq0, $
with Hurst index $H\in(0,1)$ is a centered Gaussian process with the
covariance function
\[
\ex{B^H(t)B^H(s)}=\frac{1}2
\bigl(t^{2H}+s^{2H}-\abs{t-s}^{2H} \bigr).
\]

From now on we suppose that the Hurst indices in \eqref{model} satisfy
the inequality
\begin{equation*}
\frac{1}2\leq H_1<H_2<1,
\end{equation*}
and we consider the continuous modifications of both processes, which
exist due to the Kolmogorov theorem. Assuming that the Hurst indices
$H_1$, $H_2$ and parameters $\sigma_1\geq0$, $\sigma_2\geq0$ are
known, we aim to estimate the unknown drift parameter $\theta$ by the
continuous observations of the trajectories of $X$. Due to the
long-range dependence property of fBm with $H>1/2$, we call our model
the model with double long-range dependence.

In the case where $H_1=\frac{1}2$, the problem of drift parameter
estimation in the model \eqref{model} was solved in \cite
{CaiChigKlept}, and in the case where $\frac{1}2< H_1<H_2<1$ and
$H_2-H_1>1/4$, the estimator was constructed in \cite{MiSumfbms1}. The
goal of the present paper is to generalize the results from \cite
{MiSumfbms1} to arbitrary $\frac{1}2\leq H_1<H_2<1$. The problem, more
technical than principal, is that in the case where $H_2-H_1>1/4$ and
$H_1>1/2$, the construction of the estimator is reduced to the question
if the solution of the Fredholm integral equation of the 2nd kind with
weakly singular kernel from $L_2[0,T]$ exists and is unique, but for
$H_2-H_1\leq1/4$, the kernel does not belong to $L_2[0,T]$. Moreover,
in this case, we can say that in the literature it is impossible to
pick up for this kernel any suitable standard techniques for working
with weak singular kernels, and it does not belong to any standard
class of weak singular kernels. The matter lies in the fact that the
kernel contains two power indices, $H_1$ and $H_2$, and they create
more complex singularity than it usually happens. So, it is necessary
to make many additional efforts in order to prove the compactness of
the corresponding integral operator. Immediately after establishing the
compactness of the corresponding integral operator, the problem of
statistical estimation follows the same steps as in the paper \cite
{MiSumfbms1}, and we briefly present these steps for completeness.

The paper is organized as follows. In Section~\ref{sec2}, we describe
the model and explain how to reduce the solution of the estimation
problem to the existence--uniqueness problem for the integral Fredholm
equation of the 2nd kind with some nonstandard weakly singular kernel.
In Section~\ref{sec3}, we solve the existence--uniqueness problem.
Section~\ref{sec4} is devoted to the basic properties of estimator,
that is, we establish its form, consistency, and asymptotic normality.
Section~\ref{sec5} contains the properties of hypergeometric function
used in the proof of the existence--uniqueness result for the main
Fredhom integral equation.

\section{Preliminaries. How to reduce the original problem to the
integral equation}\label{sec2} Since we suppose that the Hurst
parameters $H_1, H_2$ and scale parameters $\sigma_1,\sigma_2$ are
known, for technical simplicity, we consider the case where $\sigma
_1=\sigma_2=1$ and, as it was mentioned before, $\frac{1}2\leq
H_1<H_2<1$. If we wish to include the unknown parameter $\theta$ into
the fractional Brownian motion with the smallest Hurst parameter in
order to apply Girsanov's theorem for construction of the estimator, we
consider a couple of processes
$\{\widetilde{B}^{H_1}(t), B^{H_2}(t),\;t\ge0\}$, $i=1,2,$
defined on the space $(\varOmega,
\mathfrak{F}, (\mathfrak{F})_t)$ and let $\mathsf{P}_\theta$ be a
probability measure under which $\widetilde{B}^{H_1}$ and $B^{H_2}$
are independent, $B^{H_2}$ is a fractional Brownian motion with Hurst
parameter $H_2$, and $\widetilde{B}^{H_1}$
is a fractional Brownian motion with Hurst parameter $H_1$ and drift $
\theta$, that is,
\[
\widetilde{B}^{H_1}(t)=\theta t+ {B}^{H_1}(t).
\]

The probability measure $\mathsf{P}_0$ corresponds to the case $\theta
=0$. Our main problem is the construction of maximum likelihood
estimator for $\theta\in\R$ by the observations of the process
$Z(t)=\theta t+ B^{H_1}(t)+ B^{H_2}(t)= \widetilde{B}^{H_1}(t)+
B^{H_2}(t),\;t\in[0,T].$ As in \cite{MiSumfbms1}, we apply to $Z$ the
linear transformation in order to reduce the construction to the sum
with one term being the Wiener process. So, we take the kernel
$l_H(t,s)=(t-s)^{1/2-H}s^{1/2-H}$ and construct the integral
\begin{equation}
\begin{aligned}\label{lin.trans} Y(t)=&\int_0^tl_{H_1}(t,s)dZ(s)=
\theta\mathrm{B} \biggl(\frac{3}2-H_1,\frac{3}2-H_1
\biggr)t^{2-2H_1}+ M^{H_1}(t)
\\
&+\int_0^t l_{H_1}(t,s)dB^{H_2}(s),
\end{aligned} %
\end{equation}
where $\mathrm{B}(\alpha,\beta)=\int_0^1x^{\alpha-1}(1-x)^{\beta
-1}dx$ is the beta function, and $M^{H_1}$ is a~Gaussian martingale
(Molchan martingale), admitting the representations
\[
M^H(t) = \int_0^t
l_H(t,s)dB^H (s) = \gamma_H \int
_0^t s^{1/2-H} dW(s)
\]
with
$\gamma_H = (2H(\frac{3}2-H) \varGamma(3/2-H)^3\varGamma(H+\frac{1}2)\varGamma
(3-2H)^{-1})^{\frac{1}2}$ and a Wiener process $W$. According to \cite
{MiSumfbms1},
the linear transformation \eqref{lin.trans} is well defined, and the
processes $Z$ and $Y$ are observed simultaneously. This means that we
can reduce the original problem to the equivalent problem of the
construction of maximum likelihood estimator of $\theta\in\R$ basing
on the linear transformation~$Y$.
For simplicity, denote $\mathcal{B}_{H_1}:=\mathrm{B} (\frac{3}2-H_1,\frac{3}2-H_1 )$. Now the main problem can be formulated as
follows. Let $\frac{1}2\leq H_1<H_2<1,$
\[
\Biggl\{\widetilde{X}_1(t)=\widetilde{M}^{H_1}(t),
X_2(t):=\int_0^t
l_{H_1}(t,s)dB^{H_2}(s),\;t\ge0\Biggr\},
\]
$i=1,2$, be a couple of processes defined on the space $(\varOmega,
\mathfrak{F})$, and $\mathsf{P}_\theta$ be a probability measure
under which $\widetilde{X}_1$ and $X_2$ are independent, $B^{H_2}$ is
a~ fractional Brownian motion with Hurst parameter $H_2$, and
$\widetilde{X}_1$
is a martingale with square characteristics $\langle \widetilde
{X}_1\rangle(t)=\frac{\gamma^2_{H_1}}{2-2H_1}t^{2-2H_{1}}$ and drift
$ {\theta\mathcal{B}_{H_1}} t^{2-2H_1}$, that~is,
\[
\widetilde{X}_1(t)=\widetilde{M}^{H_1}(t)= {\theta\mathcal
{B}_{H_1}} t^{2-2H_1}+{M}^{H_1}(t).
\]

Also, denote ${X}_1(t)={M}^{H_1}(t)$. Our main problem is the
construction of maximum likelihood estimator for $\theta\in\R$ by
the observations of the process
\[
Y(t)=\theta\mathcal{B}_{H_1}t^{2-2H_1}+ X_1(t)+X_2(t)=
\widetilde {X}_1(t)+X_2(t).
\]
Note that, under the measure $\mathsf{P}_\theta$, the process
\[
\widetilde{W}(t):=W(t)+ \frac{\theta(2-2H_1)\mathcal{B}_{H_1}}{
\gamma_{H_1} (\frac{3}2-H_1 )}t^{\frac{3}2-H_1}
\]
is a Wiener process with drift. Denote $\delta_{H_1}=\frac
{(2-2H_1)\mathcal{B}_{H_1}}{ \gamma_{H_1}}$.

By Girsanov's theorem and independence of $X_1$ and $X_2$,
\begin{equation*}
\begin{aligned} \frac{dP_\theta}{dP_0}&=\exp \Biggl\{\theta
\delta_{H_1}\int_0^Ts^{\frac{1}2- H_1}d
\widetilde{W}(s) - \frac{\theta^2\delta_{H_1}^2}{4(1- H_1)}T^{2-2H_1} \Biggr\}
\\
&=\exp \biggl\{\theta\delta_{H_1}\widetilde{X}_1(T) -
\frac{\theta^2\delta_{H_1}^2}{4(1- H_1)}T^{2-2H_1} \biggr\}. \end{aligned} %
\end{equation*}

As it was mentioned in \cite{CaiChigKlept}, the derivative of such a
form is not the likelihood ratio for the problem at hand because it is
not measurable with respect to the observed $\sigma$-algebra
\[
\mathfrak{F}^Y_T:=\sigma\bigl\{Y(t),t\in[0,T]\bigr\}=
\mathfrak{F}^X_T:=\sigma \bigl\{X(t),t\in[0,T]\bigr\},
\]
where $X(t)=X_1(t)+X_2(t).$

We shall proceed as in \cite{CaiChigKlept}. Let $\mu_\theta$ be the
probability measure induced by~$Y$ on the space of continuous functions
with the supremum topology under probability $\mathsf{P}_\theta$.
Then for any measurable set $A,$
$\mu_\theta(A)=\int_A\varPhi(x)\mu_0(dx),$
where $\varPhi(x)$ is a measurable functional such that $\varPhi(X)=E_0
(\frac{dP_\theta}{dP_0} |\mathfrak{F}^X_T )$. This means
that $\mu_\theta\ll\mu_0$ for any $\theta\in\R$. Taking into
account that $\widetilde{X}_1=X_1$ under $\mathsf{P}_0$ and the fact
that the vector process $(X_1, X)$ is Gaussian, we get that the
corresponding likelihood function is given by
\begin{equation*}
\begin{gathered}L_T(X,\theta)=\mathsf{E}_0
\biggl(\frac{d\mathsf
{P}_\theta}{d\mathsf{P}_0} |\mathfrak{F}^X_T \biggr) =
\mathsf{E}_0 \biggl(\exp \biggl\{\theta\delta_{H_1}{X}_1(T)
-\frac{\theta^2\delta_{H_1}^2}{4(1- H_1)}T^{2-2H_1} \biggr\} |\mathfrak{F}^X_T
\biggr)
\\
=\exp \biggl\{\theta\delta_{H_1}\mathsf{E}_0
\bigl({X}_1(T)|\mathfrak {F}^X_T\bigr)+
\frac{\theta^2\delta_{H_1}^2}{2} \biggl(V(T)-\frac{T^{2-2H_1}}{2-2H_1} \biggr) \biggr\}, \end{gathered}
\end{equation*}
where $V(t)=\mathsf{E}_0(X_1(t)-\mathsf{E}_0(X_1(t)|\mathfrak
{F}^X_t))^2,\;t\in[0,T].$

The next reasonings repeat the corresponding part of \cite
{MiSumfbms1}. We have to solve the following problem: to find the
projection $P_X X_1(T)$ of $X_1(T)$ onto
\[
\set{X(t) = X_1(t)+X_2(t),\ t\in[0,T]}.
\]
According to \cite{jost}, the transformation formula for converting
fBm into a Wiener process is of the form
\[
W_i(t) = \int_0^t \bigl(
\bigl(K^*_{H_i}\bigr)^{-1}\ind{[0,t]} \bigr) (s)
dB^{H_i}(s),\ i=1,2,
\]
where
\[
\bigl(K^*_H f\bigr) (s) = \int_{s}^{T}
f(t) \partial_t K_H(t,s)dt = \beta_H
s^{1/2-H} \int_{s}^{T} f(t)
t^{H-1/2}(t-s)^{H-3/2}dt,
\]
$\beta_H =  (\frac{H(2H-1)}{\mathrm{B}(H-1/2,2-2H)} )^\frac{1}2, $ and the square-integrable kernel $K_H(t,s)$ is of the form
\[
K_{H}(t,s) = \beta_H s^{1/2-H}\int
_s^t (u-s)^{H-3/2}u^{H-1/2}du.
\]

We have that $W_i,\ i=1,2,$ are standard Wiener processes, which are
obviously independent. Also, we have
\begin{equation}
\label{X1} X_1(t) = \gamma_{H_1} \int
_0^t s^{1/2-H_1} dW_1(s),\;
B^{H_2}(t) = \int_0^t
K_{H_2}(t,s)dW_2(s).
\end{equation}
Then
\[
X_2(t) = \int_0^t
K_{H_1,H_2}(t,s) dW_2(s),
\]
where
\begin{equation}
\label{kernelKh1h2} K_{H_1,H_2}(t,s) = \beta_{H_2} s^{1/2-H_2}\int
_{s}^{t}(t-u)^{1/2-H_1}u^{H_2-H_1}(u-s)^{H_2-3/2}du.
\end{equation}
For an interval $[0,T]$, denote by $L^2_H[0,T]$ the completion of the
space of simple functions $f\colon[0,T]\to\R$ with respect to the
scalar product
\[
\pair{f,g}_{H}^2:= \alpha_H\int
_0^T\int_0^T
f(t)g(s)\abs{t-s}^{2H-2}dsdt,
\]
where $\alpha_H = H(2H-1)$. Note that this space contains both
functions and distributions.
For functions from $ L^2_{H_2}[0,T],$ we have that
\begin{equation*}
\int_0^T f(s)dX_2(s) = \int
_{0}^{T}\bigl(K^*_{H_1,H_2}f\bigr)
(s)dW_2(s),
\end{equation*}
where
\[
\bigl(K^*_{H_1,H_2}f\bigr) (s) = \int_s^T
f(t)\partial_t K_{H_1,H_2}(t,s)dt.
\]

The projection of $X_1(T)$ onto $\{X(t),t\in[0,T]\}$ is a centered
$X$-measurable Gaussian random variable and, therefore, is of the form
\[
P_X X_1(T) = \int_0^T
h_T(t)dX(t)
\]
with $h_T \in L^2_{H_1}[0,T]$. Note that $h_T$ still may be a
distribution. However, as we will further see, it is a continuous function.
The projection for all $u\in[0,T]$ must satisfy
\begin{equation}
\label{eq2.3} \ex{X(u) P_X X_1(T)} =
\ex{X(u)X_1(T)}.
\end{equation}
Using \eqref{eq2.3} together with independency of $X_1$ and $X_2$, we
arrive at the equation
\begin{equation}
\begin{aligned}\label{eq2.4} &\ex{X_1(u)\int
_{0}^{T}h_T(t) dX_1(t) +
X_2(u)\int_{0}^{T}h_T(t)dX_2(t)}
\\
&\quad = \ex{X_1(u)X_1(T)} = \varepsilon_{H_1}
u^{2-2H_1}, \end{aligned} %
\end{equation}
where $\varepsilon_{H} = \gamma_H^2/(2-2H)$.
Finally, from \eqref{X1}--\eqref{eq2.4} we get the prototype of
a~Fredholm integral equation
\begin{equation}
\begin{gathered}\label{main1} \varepsilon_{H_1}
u^{2-2H_1} = \gamma_{H_1}^2 \int_0^u
h_T(s)s^{1-2H_1} ds + \int_0^T
h_T(s)r_{H_1,H_2}(s,u)ds,\;u\in[0,T], \end{gathered} %
\end{equation}
where
\[
r_{H_1,H_2}(s,u)=\int_0^{s\wedge u}
\partial_s K_{H_1,H_2}(s,v)K_{H_1,H_2}(u,v)dv.
\]
Differentiating \eqref{main1}, we get the Fredholm integral equation
of the 2nd kind,
\begin{equation}
\label{integralequation} \gamma_{H_1}^2 h_T(u)u^{1-2H_1}
+ \int_0^T h_T(s) k(s,u)ds = \gamma
_{H_1}^2 u^{1-2H_1}, \quad u\in(0,T],
\end{equation}
where
\begin{equation}
\label{kernelksu} k(s,u) = \int_0^{s\wedge u}
\partial_s K_{H_1,H_2}(s,v) \partial_u
K_{H_1,H_2}(u,v)dv
\end{equation}
with the function $K_{H_1,H_2}$ defined by \eqref{kernelKh1h2}.

We will establish in Remark~\ref{lemmaker} that for the case $H_1 =
\frac{1}2$, Eq.~\eqref{integralequation} can be reduced to the
corresponding equation from \cite{CaiChigKlept}:
\begin{equation}
\label{integralequationCCK} h_T(u) + H_2(2H_2-1)\int
_0^T h_T(s)|s-u|^{2H_2-2}
ds=1, \quad u\in[0,T],
\end{equation}
but the difference between \eqref{integralequationCCK} and \eqref
{integralequation} lies in the fact that \eqref{integralequationCCK}
can be characterized as the equation with standard kernel, whereas
\eqref{integralequation} with two different power exponents is more or
less nonstandard, and, therefore, it requires an~unconventional approach.
On the one hand, it is known from the paper \cite{MiSumfbms1} that if
the conditions $H_2-H_1>\frac{1}4$ and $H_1>1/2$ are satisfied, then
Eq.~\eqref{integralequation} has a~unique solution $h_{T_n}$ with
$h_{T_n}(t)t^{\frac{1}2- H_1}\in L_2[0,T_n]$ on any sequence of
intervals $[0,T_n]$ except, possibly, a countable number of $T_n$
connected to eigenvalues of the corresponding integral operator (the
meaning of this sentence will be specified later because, finally, we
will get a similar result but in more general situation). On the other
hand, the existence--uniqueness result for Eq.~\eqref
{integralequationCCK} in~\cite{CaiChigKlept} is proved without any
restriction on Hurst index $H_2$ while $H_1= \frac{1}2$. The difference
between these results can be explained so that in \cite{CaiChigKlept}
the authors state the existence and uniqueness of the continuous
solution, whereas in \cite{MiSumfbms1} the solution is established in
the framework of $L_2 $-theory.

In this paper, we propose to consider Eq.~\eqref{integralequation} in
the space $C[0,T]$ again. This means that we consider the corresponding
integral operator as an operator from\break $C[0,T]$ into $C[0,T]$ and
establish an existence--uniqueness result in $C[0,T]$. This approach
has the advantage that we do not need anymore the assumption
$H_2-H_1>\frac{1}4$ and can include the case $H_1=1/2$ again into the
consideration.

We say that two integral equations are equivalent if they have the same
continuous solutions. In this sense, Eqs.~\eqref{main1} and \eqref
{integralequation} are equivalent, and both are equivalent to the equation
\begin{equation}
\label{integralequation2} h_T(u) +\frac{1}{\gamma_{H_1}^2} \int_0^T
h_T(s) \kappa(s,u)ds = 1, \quad u\in[0,T],
\end{equation}
with continuous right-hand side,
where
\begin{equation}
\label{kappa} \kappa(s,u) = u^{2H_1-1}k(s,u), \quad s,u\in[0,T].
\end{equation}

We get that the main problem (i.e., the MLE construction for the drift
parameter) is reduced to the existence--uniqueness result for the
integral equation~\eqref{main1}.

\section{Compactness of integral operator. Existence--uniqueness
result for the Fredholm integral equation}\label{sec3}
Consider the integral operator $K$ generated by the kernel $K$ bearing
in mind that the notations of the kernel and of the corresponding
operator will always coincide:
\[
(Kx) (u) = \int_0^T K(s,u)x(s)ds, \quad x \in
C[0,T].
\]

Now we are in position to establish the properties of the kernel
$\kappa(s,u)$ defined by \eqref{kappa} and \eqref{kernelksu}.
Introduce the notation $[0,T]^2_0=[0,T]^2\setminus\{(0,0)\}$.

\begin{lemma}\label{lemmakernel}
Up to a set of Lebesgue measure zero, the kernel $\kappa(s,u)$, $s,u
\in[0,T]$,
admits the following representation on $[0,T]:$
\begin{equation}
\label{factorization} \kappa(s,u) = %
\begin{cases}
\kappa_0(s,u)\varphi(s,u), & s\neq u, \\
0, & s=u,
\end{cases} %
\end{equation}
where $\varphi(s,u)=(s\wedge u)^{1-2H_1}u^{2H_1-1}
\vert s-u\vert^{2H_2-2H_1-1 }$, and the function $\kappa_0$ is bounded and
belongs to $C([0,T]^2_0)$.
\end{lemma}

\begin{proof} We take \eqref{kernelksu} and first present
the derivative of $K_{H_1,H_2}(t,s)$, defined by~\eqref{kernelKh1h2},
in an appropriate form. To start, put $u=s+(t-s)z$. This allows us to
rewrite $K_{H_1,H_2}(t,s)$ as
\begin{equation}
\begin{aligned}\label{new} K_{H_1,H_2}(t,s)&=
\beta_{H_2}s^{\frac{1}2 -H_2} (t-s)^{H_2-H_1}
\\
&\quad \times \int_{0}^{1}(1-z)^{\frac{1}2 - H_1}
\bigl(s + (t-s)z\bigr)^{H_2-H_1}z^{H_2-\frac{3}2}dz. \end{aligned} %
\end{equation}

Differentiating \eqref{new} w.r.t.\ $t$ for $0<s<t\leq T$, we get
\begin{equation}
\label{summ} %
\begin{aligned} \partial_t
K_{H_1,H_2}(t,s) &= (H_2-H_1)\beta_{H_2}s^{\frac{1}2-H_2}(t-s)^{H_2-H_1-1}
\\
&\quad \times\int_{0}^{1}(1-z)^{\frac{1}2 - H_1}
\bigl(s+ (t-s)z\bigr)^{H_2-H_1}z^{H_2-\frac{3}2}dz
\\
&\quad + (H_2-H_1)\beta_{H_2}s^{\frac{1}2 -H_2}
\\
&\quad \times(t-s)^{H_2-H_1}\int_{0}^{1}(1-z)^{\frac{1}2 - H_1}
\bigl(s + (t-s)z\bigr)^{H_2-H_1-1}z^{H_2-\frac{1}2}dz
\\
& = (H_2-H_1)\beta_{H_2}s^{\frac{1}2 -H_2}
(t-s)^{H_2-H_1-1}
\\
&\quad \times \Biggl(\int_{0}^{1}(1-z)^{\frac{1}2 - H_1}
\bigl(s +(t-s)z\bigr)^{H_2-H_1}z^{H_2-\frac{3}2}dz
\\
&\quad + (t-s)\int_{0}^{1}(1-z)^{\frac{1}2 - H_1}
\bigl(s + (t-s)z\bigr)^{H_2-H_1-1}z^{H_2-\frac{1}2}dz \Biggr)
\\
& =(H_2-H_1)\beta_{H_2}s^{\frac{1}2 -H_2}(t-s)^{H_2-H_1-1}
\\
&\quad \times \Biggl( s^{H_2-H_1}\int_{0}^{1}z^{H_2-\frac{3}2}(1-z)^{\frac{1}2 -H_1}
\biggl(1 - \frac{s-t}{s}z \biggr)^{H_2-H_1}dz
\\
&\quad \,{+}\, (t\,{-}\,s) {s}^{H_2-H_1-1}\!\! \int_{0}^{1}
\!(1\,{-}\,z)^{\frac{1}2 - H_1} \biggl(1 \,{-}\, \frac{s-t}{s}z
\biggr)^{H_2-H_1-1}\!z^{H_2-\frac{1}2}dz \Biggr)\!. \end{aligned} %
\end{equation}

Denote for technical simplicity $\alpha_i=H_i-\frac{1}2,\ i=1,2$. Then,
according to the definition and properties of the Gauss hypergeometric
function (see Eqs.~\eqref{gyper1} and \eqref{gyper2}), the terms in
the right-hand side of \eqref{summ} can be rewritten as follows. For
the first term, thats is, for
\begin{equation}
\label{int1}I_1(t,s):=s^{H_2-H_1}\int_{0}^{1}z^{\alpha_2-1}(1-z)^{-\alpha_1}
\biggl(1 - \frac{s-t}{s}z \biggr)^{H_2-H_1}dz,
\end{equation}
the values of parameters for the underlying integral equal $a=
H_1-H_2$, $b=\alpha_2$, $c=H_2-H_1+1,$ and $x=\frac{s-t}{s}<1$,
respectively; therefore, $\frac{x}{x-1}=\frac{t-s}{t}$, $c-b=
1-\alpha_1$,
and
\begin{align*}
I_1(t,s)&=\mathrm{B}(1-\alpha_1, \alpha_2)s^{H_2-H_1} F \biggl({H_1-H_2,\alpha_2, 1-H_1+H_2};\frac{s-t}{s} \biggr)\\
&= \mathrm{B}(1\,{-}\,\alpha_1, \alpha_2)s^{H_2-H_1}\biggl(\frac{t}{s} \biggr)^{H_2-H_1} \!F \biggl(\! {H_1\,{-}\,H_2, 1\,{-}\,\alpha_1,1\,{-}\,H_1\,{+}\,H_2}; \frac{t\,{-}\,s}{t}\! \biggr)\\
&=\mathrm{B}(1-\alpha_1, \alpha_2)t^{H_2-H_1} F\biggl({H_1-H_2,1-\alpha_1,1-H_1+H_2};\frac{t-s}{t} \biggr).
\end{align*}

Similarly, for the second term, that is, for
\begin{equation}
\label{int2}I_2(t,s):=(t-s) {s}^{H_2-H_1-1} \int
_{0}^{1}z^{\alpha_2 }(1-z)^{-\alpha_1}
\biggl(1 - \frac{s-t}{s}z \biggr)^{H_2-H_1-1}dz,
\end{equation}
the values of parameters for the underlying integral equal $a=
H_1-H_2+1$, $b=\alpha_2+1$, $c=H_2-H_1+2,$ and $x=\frac{s-t}{s}$,
respectively; therefore, $\frac{x}{x-1}=\frac{t-s}{t}$, $c-b=
1-\alpha_1$, and
\begin{equation*}
\begin{aligned} I_2(t,s)&=(t-s) {s}^{H_2-H_1-1}
\mathrm{B}(1-\alpha_1,\alpha_2+1)
\\
&\quad \times F \biggl(H_1-H_2+1,\alpha_2+1,
H_2-H_1+2;\frac{s-t}{s} \biggr)
\\
&=(t-s) {s}^{H_2-H_1-1} \biggl(\frac{t}{s} \biggr)^{H_2-H_1-1}
\\
&\quad \times B(1-\alpha_1,\alpha_2+1)F \biggl(
{H_1-H_2+1, 1-\alpha_1,2-H_1+H_2};
\frac{t-s}{t} \biggr)
\\
&=(t-s) t^{H_2-H_1-1}\mathrm{B}(1-\alpha_1,\alpha_2+1)
\\
&\quad \times F \biggl( {H_1-H_2+1, 1-
\alpha_1,2-H_1+H_2}; \frac{t-s}{t}
\biggr). \end{aligned} %
\end{equation*}
It is easy to see from the initial representations \eqref{int1} and
\eqref{int2} that $I_1(t,s)$ and $I_2(t,s)$ are continuous on the set
$0<s\leq t\leq T$.

Now, introduce the notations
\[
\varPsi_1(t,s)= \mathrm{B}(1-\alpha_1,
\alpha_2) F \biggl({H_1-H_2, 1-
\alpha_1,1-H_1+H_2}; \frac{t-s}{t}
\biggr)
\]
and
\begin{equation*}
\begin{aligned} \varPsi_2(t,s)&= \biggl(\frac{t-s}{t}
\biggr)^{1-H_2+H_1}\mathrm{B}(1-\alpha_1,\alpha_2+1)
\\
&\quad \times F \biggl( {H_1-H_2+1, 1-
\alpha_1,2-H_1+H_2}; \frac{t-s}{t}
\biggr), \end{aligned} %
\end{equation*}
so that $I_1(t,s)=t^{H_2-H_1}\varPsi_1(t,s)$ and
$I_2(t,s)=(t-s)^{H_2-H_1}\varPsi_2(t,s)$.
Note that $\frac{t-s}{t}\in[0,1)$; therefore,
\begin{equation*}
\begin{aligned} &F \biggl({H_1-H_2, 1-
\alpha_1,1-H_1+H_2}; \frac{t-s}{t}
\biggr)
\\
&\quad = \frac{1}{\mathrm{B}(1-\alpha_1, \alpha_2)}\times\int_0^1z^{-\alpha_1}(1-z)^{\alpha_2-1}
\biggl(1-\frac{t-s}{t}z \biggr)^{H_2-H_1}dz
\\
&\qquad \leq\frac{1}{\mathrm{B}(1-\alpha_1, \alpha_2)}\int_0^1z^{-\alpha_1}(1-z)^{\alpha_2-1}dz
= 1, \end{aligned} %
\end{equation*}
whence the function $\varPsi_1(t,s)$ is bounded by $\mathrm{B}(1-\alpha
_1, \alpha_2)$.
In order to establish that $\varPsi_2(t,s)$ is bounded, we use
Proposition~\ref{HGFineq2}. Its conditions are satisfied: $a =
H_1-H_2+1 \in(0,1)$, $b = 1-\alpha_1 > 0$, $c-b = \alpha_2 + 1>1$,
and $x= \frac{t-s}{t} \in[0,1)$. Therefore,
\begin{equation*}
\begin{aligned} &x^{1-H_2+H_1}F ( {H_1-H_2+1,
1-\alpha_1,2-H_1+H_2}; x ) \leq
{x^{1-H_2+H_1}}
\\
&\quad \times \biggl(1-\frac{1-\alpha_1}{1-H_1+H_2}x \biggr)^{-1-H_1+H_2}= \biggl(
\frac{1}{x}-\frac{1-\alpha_1}{1-H_1+H_2} \biggr)^{-1-H_1+H_2}
\\
&\quad \leq \biggl(1 - \frac{1-\alpha_1}{1-H_1+H_2} \biggr)^{-1-H_1+H_2} = \biggl(
\frac{1-H_1+H_2}{\alpha_2} \biggr)^{H_1-H_2+1}, \end{aligned} %
\end{equation*}
whence $\varPsi_2(t,s) \leq\mathrm{B}(1-\alpha_1,\alpha_2+1)
(\frac{1-H_1+H_2}{\alpha_2}  )^{H_1-H_2+1}$.
Additionally, both functions are homogeneous:
\[
\varPsi_i(at,as)=\varPsi_i(t,s)\; \text{for}\; a>0 ,\ i =
1,2.
\]

Introduce the notation
\begin{equation}
\label{Phifunction} \varPhi(t,s) =I_1(t,s)+I_2(t,s)=
t^{H_2-H_1}\varPsi_1(t,s) + (t-s)^{H_2-H_1}
\varPsi_2(t,s)
\end{equation}
and note that $\varPhi\in C([0,T]^2_0)$ is bounded and homogeneous:
\begin{equation}
\label{Phifunction1}\varPhi(at,as)=a^{H_2-H_1}\varPhi (t,s), a>0.
\end{equation}
In terms of notation \eqref{Phifunction}, the representation \eqref
{summ} for $\partial_t K_{H_1,H_2}(t,s)$ can be rewritten as
\begin{equation}
\partial_t K_{H_1,H_2}(t,s) = \beta_{H_2}(H_2-H_1)s^{\frac{1}2
-H_2}(t-s)^{H_2-H_1-1}
\varPhi(t,s).
\end{equation}

In turn, the kernel $k(s,u)$ from \eqref{kernelksu} can be rewritten as
\begin{equation}
\begin{aligned} k(s,u) &= \bigl(\beta_{H_2}(H_2-H_1)
\bigr)^2
\\
&\quad \times\int_{0}^{s\wedge u} v^{1-2H_2}(s-v)^{H_2-H_1-1}(u-v)^{H_2-H_1-1}
\varPhi(s,v) \varPhi(u,v)dv. \end{aligned} %
\end{equation}

Consider the kernel $k(s,u)$ for $s>u$. Then it evidently equals
\begin{equation*}
\begin{aligned} k(s,u) &= \bigl(\beta_{H_2}(H_2-H_1)
\bigr)^2
\\
&\quad \times\int_{0}^{u} v^{1-2H_2}(s-v)^{H_2-H_1-1}(u-v)^{H_2-H_1-1}
\varPhi(s,v) \varPhi(u,v)dv. \end{aligned} %
\end{equation*}

Put $z = \frac{u-v}{s-u}$ and transform $k(s,u)$ to
\begin{equation*}
\begin{aligned} k(s,u)&=\bigl(\beta_{H_2}(H_2-H_1)
\bigr)^2(s-u)^{2H_2-2H_1-1} \int_0^{\frac{u}{s-u}}z^{H_2-H_1-1}(1+z)^{H_2-H_1-1}
\\
&\quad \times\bigl(u-z(s-u)\bigr)^{1-2H_2}\varPhi\bigl(s,u-z(s-u)\bigr)\varPhi
\bigl(u,u-z(s-u)\bigr) dz
\\
&=: \frac{k_0(s,u)}{(s-u)^{1-2H_2+2H_1}}, \end{aligned} %
\end{equation*}
where
\begin{equation*}
\begin{aligned} k_0(s,u)&=\bigl(\beta_{H_2}(H_2-H_1)
\bigr)^2\int_0^{\frac{u}{s-u}}z^{H_2-H_1-1}(1+z)^{H_2-H_1-1}
\\
&\quad \times\bigl(u-z(s-u)\bigr)^{1-2H_2}\varPhi\bigl(s,u-z(s-u)\bigr)\varPhi
\bigl(u,u-z(s-u)\bigr) dz. \end{aligned} %
\end{equation*}

In turn, transform $k_0(s,u)$ with the change of variables $tu= z$ and
apply~\eqref{Phifunction1}:
\begin{align}
k_0(s,u)&=\bigl(\beta_{H_2}(H_2-H_1)
\bigr)^2\int_0^{\frac{1}{s-u}}(tu)^{H_2-H_1-1}(1+tu)^{H_2-H_1-1}\nonumber\\
&\quad \times\bigl(u-tu(s-u)\bigr)^{1-2H_2}\varPhi\bigl(s,u-tu(s-u)\bigr)\varPhi
\bigl(u,u-tu(s-u)\bigr)udt\nonumber\\
&= \bigl(\beta_{H_2}(H_2-H_1)
\bigr)^2 u^{1-2H_1}\int_0^{\frac{1}{s-u}}
\bigl(1-t(s-u)\bigr)^{1-2H_2}(1+tu)^{H_2-H_1-1}\nonumber\\
&\quad \times t^{H_2-H_1-1}\varPhi\bigl(s,u-tu(s-u)\bigr)\varPhi\bigl(1,1-t(s-u)
\bigr)dt.
\end{align}

Introducing the kernel $\kappa_0(s,u)=k_0(s,u)u^{2H_1-1}$, we can
present $k(s,u)$ as
\begin{equation}
\label{kernelksupresent} k(s,u) = \frac{\kappa_0(s,u)}{(s-u)^{1-2H_2+2H_1} u^{2H_1-1}},
\end{equation}
where, for $s>u>0,$
\begin{align}
\kappa_0(s,u) &= \bigl(
\beta_{H_2}(H_2-H_1)\bigr)^2 \int
_{0}^{\frac{1}{s-u}}\bigl(1-(s-u)t\bigr)^{1-2H_2}(1+ut)^{H_2-H_1-1}\nonumber\\
&\quad \times t^{H_2-H_1-1}\varPhi\bigl(s,u-tu(s-u)\bigr)\varPhi\bigl(1,1-t(s-u)
\bigr) dt\nonumber\\
&=\bigl(\beta_{H_2}(H_2-H_1)
\bigr)^2 \int_{0}^{\infty} 1_{t\leq\frac{1}{s-u}}
\bigl(1-(s-u)t\bigr)^{1-2H_2}(1+ut)^{H_2-H_1-1}\nonumber\\
&\quad \times t^{H_2-H_1-1}\varPhi\bigl(s,u-tu(s-u)\bigr)\varPhi\bigl(1,1-t(s-u)
\bigr) dt.\label{kappasuinit}
\end{align}

For the case $u > s>0,$ we can replace $s$ and $u$ in formulas \eqref
{kernelksupresent} and \eqref{kappasuinit}.
Substituting formally $u=s$ into \eqref{kappasuinit}, for $s>0,$ we get
\begin{equation}
\begin{aligned}\label{kappasuinit1} \kappa_0(s,s) &= \bigl(
\beta_{H_2}(H_2-H_1)\bigr)^2\varPhi(s,s)
\varPhi(1,1 ) \int_{0}^{\infty} (1+st)^{H_2-H_1-1}
t^{H_2-H_1-1}dt
\\
&=\bigl(\beta_{H_2}(H_2-H_1)
\bigr)^2s^{H_2-H_1}\varPhi(1,1 )^2\int
_{0}^{\infty} (1+st)^{H_2-H_1-1} t^{H_2-H_1-1}dt.
\end{aligned} %
\end{equation}
Note that $\varPhi(1,1) =\mathrm{B}(1-\alpha_1,\alpha_2)$ and $\int_{0}^{\infty} (1+st)^{H_2-H_1-1} t^{H_2-H_1-1}
dt=s^{H_1-H_2}\allowbreak \mathrm{B}(H_2-H_1,1-2H_2+H_1)$. The former equation
holds due to \eqref{HGFatzero}. We get that $\kappa_0(s,s)$ does not
depend on $s$ and equals some constant $C_H:=(\beta
_{H_2}(H_2-H_1)\mathrm{B}(1-\alpha_1,\alpha_2))^2\mathrm
{B}(H_2-H_1,1-2H_2+H_1)$. Therefore, we define $\kappa_0(s,s)=C_H,\ s>0.$

Now the continuity of $\kappa_0 $ on $(0,T]^2 $ follows from the
Lebesgue dominated convergence theorem supplied by representation
\eqref{kappasuinit}, Eq.~\eqref{kappasuinit1}, and its consequence
$\kappa_0(s,s)=C_H,\ s>0 $, together with the facts that $\varPhi\in
C([0,T]^2_0)$ and is bounded. Consider $\kappa_0(s,u)$ for $u
\downarrow0$ and let $s>0$ be fixed. Then
\begin{equation*}
\begin{aligned} \lim\limits
_{u\downarrow0}\kappa_0(s,u) &=
C_H^1:=\bigl(\beta_{H_2}(H_2-H_1)
\bigr)^2 \varPhi(1,0)
\\
&\quad \times\int_{0}^{1}(1-y)^{1-2H_2}y^{H_2-H_1-1}
\varPhi(1,1-y)dy < \infty, \end{aligned} %
\end{equation*}
and we can put $\kappa_0(s,0)=\kappa_0(0,u)=C_H^1,\ s>0,u>0,$ thus
extending the continuity of $\kappa_0$ to $[0,T]^2_0$.

It is easy to see that the values $\kappa_0(s,s)$ and $\kappa_0(s,0)$
do not depend on $s>0$ and do not coincide: $C_H \neq C_{H}^1$.
Consequently, the limit
\[
\lim\limits
_{(s,u)\rightarrow(0,0)}\kappa_0(s,u)
\]
does not exist and depends on the way the variables $s$ and $u$ tend to
zero. We can equate $\kappa_0(0,0) $ to any constant; for example, let
$\kappa_0(0,0)=0$.

In order to prove that $\kappa_0$ is bounded, we consider the case
$s>u$ (the opposite case is treated similarly) and put $z = (s-u)t$. Then
\begin{equation}
\label{int3} %
\begin{aligned} &\int_{0}^{\frac{1}{s-u}}
\bigl(1-(s-u)t\bigr)^{1-2H_2}(1+ut)^{H_2-H_1-1}t^{H_2-H_1-1}\varPhi
\bigl(s,u-tu(s-u)\bigr)
\\
&\quad \times\varPhi\bigl(1,1-t(s-u)\bigr) dt = \frac{1}{(s-u)^{H_2-H_1}} \int
_{0}^{1} (1-z)^{1-2H_2}
\\
&\quad \times \biggl(1+\frac{u}{s-u}z \biggr)^{H_2-H_1-1}z^{H_2-H_1-1}
\varPhi\bigl(s,u(1-z)\bigr))\varPhi(1,1-z) dz=:I_3(s,u). \end{aligned}
\end{equation}
It follows from \eqref{Phifunction1} that, for $s\neq0,$
\[
\varPhi\bigl(s,u(1-z)\bigr)=s^{H_2-H_1}\varPhi \biggl(1,\frac{u}{s}(1-z)
\biggr).
\]
Denote $r = \frac{s}{s-u}$ and put $t = \frac{1-z}{1-(1-r)z}$. Then
\[
\frac{u}{s-u} = r - 1,\;t<1,\; z= \frac{1-t}{1-t(1-r)}\in(0,1),
\]
and the right-hand side of \eqref{int3} can be rewritten as
\begin{equation}
\label{int4} %
\begin{aligned} I_3(s,u)&= r^{H_2-H_1}
\int_{0}^{1}(1-z)^{1-2H_2}\bigl(1-(1-r)z
\bigr)^{H_2-H_1-1}z^{H_2-H_1-1}
\\
&\quad \times\varPhi \biggl(1,\frac{u}{s}(1-z) \biggr)\varPhi(1,1-z) dz =
r^{1-2H_1}\int_{0}^{1} t^{1-2H_2}(1-t)^{H_2-H_1-1}
\\
&\quad \times \bigl(1-(1-r)t\bigr)^{2H_1-1}\varPhi \biggl(1,
\frac{u}{s} \frac{rt}{1-(1-r)t} \biggr)\varPhi \biggl(1,\frac{rt}{1-(1-r)t}
\biggr) dt. \end{aligned} %
\end{equation}

Finally, put $y=1-t$. Then the right-hand side of \eqref{int4} is
transformed to
\begin{equation*}
\begin{aligned} I_3(s,u)&=r^{1-2H_1}r^{2H_1-1}
\int_{0}^{1}(1-y)^{1-2H_2}y^{H_2-H_1-1}
\biggl(1-y\frac{r-1}{r} \biggr)^{2H_1-1}
\\
&\quad \times\varPhi \biggl(1,\frac{u}{s} \frac{r(1-y)}{r-y(r-1)} \biggr)\varPhi
\biggl(1,\frac{r(1-y)}{r-y(r-1)} \biggr) dy. \end{aligned} %
\end{equation*}

Recall that $r=\frac{s}{s-u}$. Then it follows from the boundedness of
$\varPhi$ that there exists a constant $C_H^1$ such that, for $s>u$,
\begin{align}
\kappa_0(s,u) &= \bigl(
\beta_{H_2}(H_2-H_1)\bigr)^2 \int
_{0}^{1}(1-y)^{1-2H_2} \biggl(1-
\frac{u}{s}y \biggr)^{2H_1-1}y^{H_2-H_1-1}\nonumber\\
&\quad \times\varPhi \biggl(1, \frac{u(1-y)}{s-uy} \biggr)\varPhi \biggl(1,
\frac{s(1-y)}{s-uy} \biggr)dy\nonumber\\
&\quad \leq C_H^1\int_{0}^{1}
(1-y)^{1-2H_2}y^{H_2-H_1-1}dy,\label{kappasu}
\end{align}
so $\kappa_0$ is bounded, and the lemma is proved.
\end{proof}

\begin{remark} Figure~\ref{kappagraph} demonstrates the graph of
$\kappa_0(s,u)$ for $H_1=0.7$ and $H_2=0.9$.

\begin{figure}[t]
\includegraphics{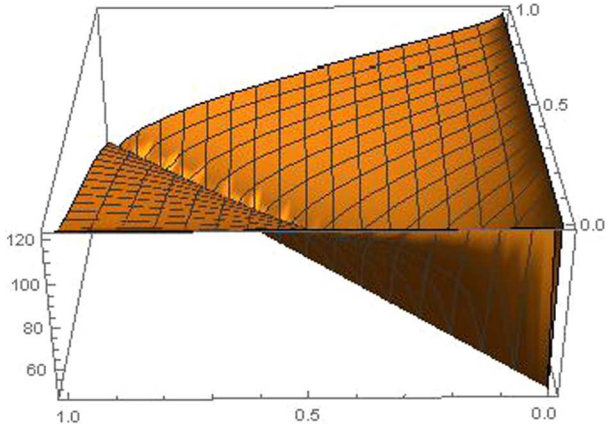}
\caption{Function $\kappa_0(s,u)$}
\label{kappagraph}
\end{figure}
\end{remark}
Now, consider the properties of the function
\[
\varphi(s,u)=(s\wedge u)^{1-2H_1}u^{2H_1-1} \abs{s-u}^{2H_2-2H_1-1 }
\]
participating in the kernel representation \eqref{factorization}.

\begin{lemma}\label{lemmavarphi}
The function $\varphi$ has the following properties:
\medskip
\begin{enumerate}\addtolength{\itemsep}{0.7\baselineskip}
\item[\rm(i)] \label{pi} for any $u\in[0,T]$, $\varphi(\cdot
,u)\in L_1[0,T]$ and $\sup\limits_{u\in[0,T]}
\lVert\varphi(\cdot,u)\rVert_{L_1} < \infty$

\item[\rm(ii)] \label{pii} for any $u_1\in[0,T]$, $\int_0^T
\vert\varphi(s,u) - \varphi(s,u_1)\vert ds \rightarrow0$ as $u\rightarrow u_1$.
\end{enumerate}
\end{lemma}

\begin{proof}
$(i)$ It follows from the evident calculations that
\begin{equation*}
\begin{aligned} &\int_0^T \abs{
\varphi(s,u)}ds = \int_0^T \varphi(s,u)ds = \int
_0^u \frac{u^{2H_1-1}ds}{s^{2H_1-1}(u-s)^{1+2H_1-2H_2}}
\\
&\quad +\int_u^T \frac{ds}{(s-u)^{1+2H_1-2H_2}} =
u^{2H_2-2H_1}\mathrm{B}(2-2H_1,2H_2-2H_1)
\\
&\quad + \frac{(T-u)^{2H_2-2H_1}}{2H_2-2H_1} \leq C_{H_1,H_2}T^{2H_2-2H_1} < \infty\text{
for all } u \in[0,T]. \end{aligned} %
\end{equation*}

(ii) First, let $u_1 = 0$ and $u\downarrow0$. Note that $\varphi
(s,0)=s^{2H_2-2H_1-1 }.$ Therefore,
\begin{equation*}
\begin{aligned} &\int_0^T \abs{
\varphi(s,u) - \frac{1}{s^{1+2H_1-2H_2}}}ds = \int_0^u
\abs{\frac{u^{2H_1-1}}{s^{2H_1-1}(u-s)^{1+2H_1-2H_2}} -\frac{1}{s^{1+2H_1-2H_2}}}ds
\\
&\qquad \,{+} \int_u^T\! \frac{ds}{(s-u)^{1+2H_1-2H_2}}
\,{-} \int_u^T\! \frac{ds}{s^{1+2H_1-2H_2}}ds \,{\leq}
\int_0^u\! \frac{u^{2H_1-1}ds}{s^{2H_1-1}(u-s)^{1+2H_1-2H_2}}
\\
&\qquad + \int_0^u \frac{ds}{s^{1+2H_1-2H_2}} +
\frac{1}{2H_2-2H_1}\bigl((s-u)^{2H_2-2H_1} - s^{2H_2-2H_1}\bigr)
\vert_{s=u}^{s=T}
\\
&\quad =\mathrm{B}(2-2H_1, 2H_2-2H_1)
u^{2H_2-2H_1}
\\
&\qquad + \frac{1}{2H_2-2H_1}\bigl(2u^{2H_2-2H_1} + (T-u)^{2H_2-2H_1} -
T^{2H_2-2H_1}\bigr) \rightarrow0, \text{ as } u \rightarrow0. \end{aligned}
\end{equation*}

From now on suppose that $u_1 > 0$ is fixed. Without loss of
generality, suppose that $u\uparrow u_1$. Then
\begin{equation*}
\begin{aligned} &\int\limits
_0^T \abs{\varphi(s,u) -
\varphi(s,u_1)}ds = \int\limits
_0^{u}\abs{\varphi(s,u)-
\varphi(s,u_1)}ds+\int\limits
_{u}^{u_1}\abs{\varphi(s,u)-
\varphi(s,u_1)}ds
\\
&\quad +\int_{u_1}^T\abs{\varphi(s,u)-
\varphi(s,u_1)}ds =: I_1(u, u_1) +
I_2(u,u_1) + I_3(u, u_1).
\end{aligned} %
\end{equation*}

Consider the terms separately.
First, we establish that $\varphi(s,\cdot)$ is decreasing in the
second argument. Indeed, for
$0<s<u < u_1$,
\begin{equation*}
\begin{aligned} \varphi(s,u_1) &= \frac{u_1^{2H_1-1}}{s^{2H_1-1}(u_1-s)^{1+2H_1-2H_2}}=
\frac{1}{s^{2H_1-1}(1-\frac{s}{u_1})^{1+2H_1-2H_2}u_1^{2-2H_2}}
\\
&\quad \leq\frac{1}{s^{2H_1-1}(1-\frac{s}{u})^{1+2H_1-2H_2}u^{2-2H_2}} = \varphi(s,u). \end{aligned} %
\end{equation*}
Therefore,
\begin{equation*}
\begin{aligned} I_1(u, u_1) &= \int
_0^{u} \bigl(\varphi(s,u)-\varphi(s,u_1)
\bigr)ds = \int_0^{u}\varphi(s,u)ds - \int
_0^{u_1}\varphi(s,u_1)ds
\\
&\quad + \int\limits
_{u}^{u_1}\varphi(s,u_1)ds \leq
\mathrm{B}(2-2H_1,2H_2-2H_1)
\bigl(u^{2H_2-2H_1} - u_1^{2H_2-2H_1}\bigr)
\\
&\quad + \frac{u_1^{2H_1-1}(u_1-u)^{2H_2-2H_1}}{2H_2-2H_1} \rightarrow0, \text{ as } u \uparrow
u_1. \end{aligned} %
\end{equation*}
The second integral vanishes as well:
\begin{equation*}
\begin{aligned} I_2(u, u_1) &\leq\int\limits
_{u}^{u_1}
\varphi(s,u)ds+ \int\limits
_{u}^{u_1}\varphi(s,u_1)ds
\\
&\leq \biggl(\frac{1}{2H_2-2H_1} + \biggl(\frac{u_1}{u}
\biggr)^{2H_1-1} \biggr) (u_1-u)^{2H_2-2H_1}\rightarrow0
\end{aligned} %
\end{equation*}
as $u \uparrow u_1$. Finally,
\begin{equation*}
\begin{aligned} I_3(u, u_1) &= \int
_{u_1}^T\frac{ds}{(s-u_1)^{1+2H_1-2H_2}}-\int_{u_1}^T
\frac{ds}{(s-u)^{1+2H_1-2H_2}} = \frac{1}{2H_2-2H_1}
\\
&\quad \times\bigl((T-u_1)^{2H_2-2H_1}-(T-u)^{2H_2-2H_1}+(u_1-u)^{2H_2-2H_1}
\bigr)\rightarrow0
\\
&\quad \qquad \qquad \quad \quad \quad \quad \quad \text{ as } u \uparrow
u_1. \end{aligned} %
\end{equation*}
The lemma is proved.
\end{proof}

\begin{lemma}\label{compactKernel}
The kernel $\kappa$ generates a compact integral operator $\kappa
:C[0,T] \rightarrow C[0,T]$.
\end{lemma}

\begin{proof}
According to \cite{BeresUS}, it suffices to prove that the kernel
$\kappa$ defined by \eqref{factorization} satisfies the following two
conditions:

\begin{itemize}
\item[(iii)] for any $u\in[0,T]$, $\kappa(\cdot,u)\in L_1[0,T]$ and
$\sup\limits_{u\in[0,T]}\lVert\kappa(\cdot,u)\rVert_{L_1} < \infty$;
\item[ (iv)] For any $u_1 \in[0,T]$, $\int_0^T
\vert\kappa(s,u)-\kappa(s,u_1)\vert ds \rightarrow0$ as $u \rightarrow u_1$.
\end{itemize}
The first condition follows directly from fact that $\kappa_0(s,u)$ is
bounded (see Lemma~\ref{lemmakernel}) and from Lemma~\ref
{lemmavarphi} (i).

In order to check (iv), consider
\begin{equation*}
\begin{aligned} &\int_0^T\abs{
\kappa(s,u)-\kappa(s,u_1)}ds = \int_0^T
\abs{\kappa_0(s,u)\varphi(s,u)-\kappa_0(s,u_1)
\varphi(s,u_1)}ds
\\
&\quad \leq\int_0^T\kappa_0(s,u)
\abs{\varphi(s,u)-\varphi(s,u_1)}ds+\int_0^T
\varphi(s,u_1) \abs{\kappa_0(s,u)-\kappa_0(s,u_1)}ds.
\end{aligned} %
\end{equation*}
Again, Lemma~\ref{lemmakernel} in the part that states that $\kappa
_0(s,u)$ is bounded, together with Lemma~\ref{lemmavarphi} (ii),
guarantees that the first term converges to zero as $u\rightarrow u_1$.
Furthermore, Lemma~\ref{lemmakernel} in the part that states that
$\kappa_0\in C([0,T]^2_0) $ guarantees that $\kappa_0(s,u)$ converges
to $\kappa_0(s,u_1)$ as $u\rightarrow u_1$ for a.e.\ $s\in[0,T]$. Since
\[
\varphi(s,u_1) \abs{\kappa_0(s,u)-\kappa_0(s,u_1)}
\leq C\varphi (s,u_1)\in L_1[0,T],
\]
the proof follows from the Lebesgue dominated convergence theorem.
\end{proof}

\begin{remark}\label{lemmaker}
In the case where $H_1=\frac{1}2$, the kernel $\kappa(s,u)$ can be
simplified to
\[
\kappa(s,u)= H_2(2H_2-1)\abs{s-u}^{2H_2-2},
\]
and Eq.~\eqref{integralequation} coincides with \eqref
{integralequationCCK}. Indeed,
let $H_1=\frac{1}2$. Then the function $\kappa_0(s,u)$ equals
$H_2(2H_2-1)$. Consider the function $\varPhi(s,v)$ defined by \eqref
{Phifunction}:
\begin{equation}
\begin{aligned}\label{Phisimple} \varPhi(t,s) &=t^{H_2-\frac{1}2} \Biggl(
\int_{0}^{1}\biggl(1 - \frac{t-s}{t}z
\biggr)^{H_2-\frac{1}2}(1-z)^{H_2-\frac{3}2}dz
\\
&\quad + \frac{t-s}{t}\int_0^1(1-z)^{H_2-\frac{1}2}
\biggl(1-\frac{t-s}{t}z\biggr)^{H_2-\frac{3}2}dz \Biggr)
\\
&= -\frac{t^{H_2-\frac{1}2}}{H_2-\frac{1}2}\int_{0}^{1} \biggl(
\biggl(1 - \frac{t-s}{t}z\biggr)^{H_2-\frac{1}2}(1-z)^{H_2-\frac{1}2}
\biggr)^{'}_{z}dz = \frac{t^{H_2-\frac{1}2}}{H_2-\frac{1}2}. \end{aligned}
\end{equation}

Combining \eqref{kappasu} and \eqref{Phisimple}, we get
\begin{equation*}
\begin{aligned}
&\kappa_0(s,u)\\
&\quad  = \bigl(\beta_{H_2}(H_2\,{-}\,H_1)\bigr)^2 \!\int_{0}^{1}(1-t)^{1-2H_2}t^{H_2-\frac{3}2}\varPhi\biggl(1,\frac{u(1-t)}{s-ut}\biggr)\varPhi\biggl(1,\frac{s(1-t)}{s-ut}\biggr)dt\\
&\quad = \beta_{H_2}^2\!\int_{0}^{1}\! (1\,{-}\,t)^{1-2H_2}t^{H_2-\frac{3}2}dt \,{=}\, \beta_{H_2}^2\mathrm{B}\biggl(H_2\,{-}\,\frac{1}2,2-2H_2\biggr) = H_2(2H_2-1).
\end{aligned} %
\end{equation*}
\end{remark}

\begin{thm}
There exists a sequence $T_n \rightarrow\infty$ such that
the integral equation \eqref{integralequation2} has a unique solution
$h_{T_n}(u)\in C[0,T_n]$.
\end{thm}

\begin{proof} We work on the space $C([0,T])$.
Recall that \eqref{integralequation2} is of the form
\[
h_T(u) +\frac{1}{\gamma_{H_1}^2} \int_0^T
h_T(s) \kappa(s,u)ds = 1, \quad u\in[0,T].
\]
The corresponding homogeneous equation is of the form
\begin{equation}
\label{homogeneous} \int_0^T h_T(s)
\kappa(s,u)ds = -\gamma_{H_1}^2 h_T(u), \quad u
\in[0,T].
\end{equation}

Since the integral operator $\kappa$ is compact, classical Fredholm
theory states that Eq.~\eqref{integralequation2} has a unique solution
if and only if the corresponding homogeneous equation \eqref
{homogeneous} has only the trivial solution.
Now, it is easy to see that, for any $a>0$, the following equalities hold:
\begin{align*}
\kappa_0(sa,ua) &= \kappa_0(s,u),
\\
\varphi(sa,ua) &= a^{2H_2-2H_1-1}\varphi(s,u).
\end{align*}

Consequently, $\kappa(sa,ua) = a^{2H_2-2H_1-1}\kappa(s,u).$ We can
change the variable of integration $s=s' T$ and put $u=u'T$ in \eqref
{homogeneous}. Therefore, the equation will be reduced to the
equivalent form
\begin{equation*}
\int_0^1 h_T(Ts) \kappa(s,u)ds = -
\gamma_{H_1}^2 T^{2H_1-2H_2} h_T(Tu), \quad u
\in[0,1].
\end{equation*}
Denote $\lambda= -\gamma_{H_1}^2 T^{2H_1-2H_2}$. Note that $\lambda$
depends continuously on $T$. At the same time, the compact operator
$\kappa$ has no more than countably many eigenvalues. Therefore, we
can take the sequence $T_n\rightarrow\infty$ in such a way that
\[
\lambda_n = -\gamma_{H_1}^2
T_n^{2H_1-2H_2}
\]
will be not an eigenvalue. Consequently, the homogeneous equation has
only the trivial solution, whence the proof follows.
\end{proof}

\section{Statistical results: The form of a maximum likelihood
estimator, its consistency, and asymptotic normality}\label{sec4}

The following result establishes the way MLE for the drift parameter
$\theta$ can be calculated. The proof of the theorem is the same as
the proof of the corresponding statement from \cite{MiSumfbms1}, so we
omit it.

\begin{thm} The likelihood function is of the form
\begin{equation*}
\begin{gathered}L_{T_n}(X,\theta)=\exp\biggl\{\theta\delta
_{H_1}N(T_n)-\frac{1}2\theta^2
\delta^2_{H_1}\langle N\rangle(T_n)\biggr\},
\end{gathered} %
\end{equation*}
and the maximum likelihood estimator is of the form
\begin{equation*}
\begin{gathered}\widehat{\theta}(T_n)=
\frac{N(T_n)}{\delta
_{H_1}\langle N\rangle(T_n)}, \end{gathered} %
\end{equation*}
where $N(t)=E_0(X_1(t)|\mathfrak{F}^X_t)$ is a square-integrable
Gaussian $\mathfrak{F}^X_t$-martingale,\\ $N(T_n)=\int_0^{T_n}h_{T_n}(t)dX(t)$ with $h_{T_n}(t)t^{\frac{1}2-H_1}\in
L_2[0,T_n]$, $h_{T_n}(t)$ is a unique solution to \eqref
{integralequation2}, and $\langle N\rangle(T_n)=\gamma^2_{H_1}\int_0^{T_n} h_{T_n}(t)t^{1-2H_1}dt.$
\end{thm}

The next two results establish basic properties of the estimator; their
proofs repeat the proofs of the corresponding statements from \cite
{MiSumfbms1} and \cite{CaiChigKlept}.

\begin{thm}
The estimator $\widehat{\theta}_{T_n}$ is strongly consistent, and
\[
\lim_{{T_n}\rightarrow\infty}{T_n}^{2-2H_2}E_\theta(
\widehat {\theta}_{T_n}-\theta)^2=\frac{1}{\int_0^1h_0(u)u^{\frac{1}2-H_1}du},
\]
where the function $h_0(u)$ is the solution of the integral equation
\begin{equation*}
\kappa h(u) = \gamma^2_{H_1}.
\end{equation*}
\end{thm}

\begin{thm}
The estimator $\widehat{\theta}_{T_n}$ is unbiased, and the
corresponding estimation error is normal
\[
\widehat{\theta}_{T_n}-\theta\sim N \biggl(0, \frac{1}{\int_0^{T_n}h_{T_n}(s)s^{1-2H_1}ds}
\biggr).
\]
\end{thm}

\appendix

\section{Appendix. Some properties of the hypergeometric function}\label{sec5}

Recall the integral representation of the Gauss hypergeometric function
and some of its properties.

For $c> b>0$ and $x < 1,$ the Gauss hypergeometric function is defined
as the integral (see \cite{Abr}, formula 15.3.1)
\begin{equation}
\begin{gathered}\label{gyper1} F ({a, b, c}; x )={}_2F_1
({a, b, c}; x ) = \frac
{1}{\mathrm{B}(b,c-b)}\int_0^1t^{b-1}(1-t)^{c-b-1}(1-xt)^{-a}dt.
\end{gathered} %
\end{equation}

For the same values of parameters, the following equality holds (see
\cite{Abr}, 15.3.4):
\begin{equation}
\label{gyper2} F ( a, b, c; x ) = (1-x)^{-a} F \biggl( {a, c-b,c};
\frac
{x}{x-1} \biggr),
\end{equation}

Evidently, $F ( {a, b, c}; x )$ at $x=1$ is correctly defined
for $c-a-b>1$ and in this case equals
\begin{equation}
\label{eq2.25} F ( {a, b, c}; 1 ) = \frac{\varGamma(c)\varGamma(c-a-b)}{\varGamma
(c-a)\varGamma(c-b)}.
\end{equation}

Finally, it is easy to check with the help of \eqref{gyper1} that
\begin{equation}
\label{HGFatzero} F ( {a, b, c}; 0 )=F ( {0, b, c}; x ) = 1.
\end{equation}

The following result gives upper bounds for the hypergeometric function
(see \cite{HGF2} Theorem 4 and 5, respectively).

\begin{prop} \label{HGFineq2}
{\rm(i)} For $c>b>1,\ x>0$, and $0<a\leq1$, we have the inequality
\[
F ({a, b, c}; -x ) < \frac{1}{(1+x(b-1)/(c-1))^a}.
\]

{\rm(ii)} For $0<a\leq1,\ b>0,\ c-b >1, $ and $x\in(0,1),$ we have
the inequality
\[
F ( {a, b, c}; x ) < \frac{1}{(1-\frac{b}{c-1}x)^{a}}.
\]
\end{prop}

\end{document}